\documentclass[11pt,a4paper]{article}

\usepackage{soul}
\usepackage{tikz}
\usepackage{amsthm}
\usepackage{amsmath}
\usepackage{amssymb}
\usepackage{mathrsfs}

\usepackage{geometry}
\usepackage{graphicx}
\usepackage{amsfonts}
\usepackage{epstopdf}
\usepackage{placeins}
\usepackage{enumerate}
\usepackage{enumitem}
\usepackage{subcaption}
\usepackage{microtype}
\usepackage[hidelinks]{hyperref} % avoid too many boxes with colour
\usepackage[normalem]{ulem} % enables \sout, avoids changing \emph 

\usetikzlibrary{calc,arrows.meta, positioning}

\numberwithin{equation}{section}

\allowdisplaybreaks

\newcommand{\esup}{\operatorname*{ess\,sup}}
\newcommand{\einf}{\operatorname*{ess\,inf}}
\newcommand{\eosc}{\operatorname*{ess\,osc}}

\newcommand{\Rmnum}[1]{\uppercase\expandafter{\romannumeral#1}} % Uppercase roman number

\def\Xint#1{\mathchoice
{\XXint\displaystyle\textstyle{#1}}%
{\XXint\textstyle\scriptstyle{#1}}%
{\XXint\scriptstyle\scriptscriptstyle{#1}}%
{\XXint\scriptscriptstyle\scriptscriptstyle{#1}}%
\!\int}
\def\XXint#1#2#3{{\setbox0=\hbox{$#1{#2#3}{\int}$ }
\vcenter{\hbox{$#2#3$ }}\kern-.6\wd0}}

\def\dashint{\Xint-}

\theoremstyle{plain}
\newtheorem{theorem}{Theorem}[section]
\newtheorem{proposition}[theorem]{Proposition}
\newtheorem{lemma}[theorem]{Lemma}
\newtheorem{corollary}[theorem]{Corollary}

\theoremstyle{definition}
\newtheorem{remark}[theorem]{Remark}

\renewcommand{\thefootnote}{}

\AtEndDocument{
\bigskip{\footnotesize%
   \textsc{Department of Mathematics, Aarhus University, 8000 Aarhus C, Denmark} \par
  \textit{E-mail address}: \texttt{yangmengqh@gmail.com}\par
  \textit{E-mail address}: \texttt{yang@math.au.dk}
}}

\begin{document}

\title{\normalsize{Elliptic Harnack inequality and Poincar\'e inequality for $p$-energies on metric measure spaces}}
\author{Meng Yang}
\date{}

\maketitle

\abstract{For $p>1$ and a $p$-energy on a volume doubling metric measure space, we prove the Poincar\'e inequality under the elliptic Harnack inequality, two-sided capacity bounds, and additional geometric and analytic assumptions. Combining this with our previous results, we obtain, under the same structural assumptions, the equivalence between the elliptic Harnack inequality together with two-sided capacity bounds and the conjunction of the Poincar\'e inequality and the cutoff Sobolev inequality.}

\footnote{\textsl{Date}: \today}
\footnote{The author is very grateful to Aobo Chen for helpful discussions.}

\renewcommand{\thefootnote}{\arabic{footnote}}
\setcounter{footnote}{0}

\section{Introduction}\label{sec_intro}

For strongly local regular Dirichlet forms, the relationship among heat kernel estimates, Harnack inequalities, Poincar\'e inequalities, cutoff Sobolev inequalities, and capacity bounds is by now well understood; see, for example, \cite{BBK06,GT12,GH14,AB15,GHL15,Eri26} and the references therein. In particular, for a strongly local regular Dirichlet form on a volume doubling metric measure space, the elliptic Harnack inequality together with two-sided capacity bounds is equivalent to the conjunction of the Poincar\'e inequality and the cutoff Sobolev inequality:
$$\text{EHI}+\text{cap}(\Psi)\Leftrightarrow\text{PI}(\Psi)+\text{CS}(\Psi),$$
where $\Psi$ is a suitable general scaling function. This equivalence plays a fundamental role in the stability theory for the elliptic Harnack inequality; see \cite{BM18,KM23}.

For nonlinear $p$-energies, there is generally neither a linear semi-group nor a heat kernel. Nevertheless, many functional inequalities still make sense. In particular, one may formulate $p$-versions of the elliptic Harnack inequality, the Poincar\'e inequality, the cutoff Sobolev inequality, and the two-sided capacity bounds. It is therefore natural to ask whether the above equivalence has a nonlinear counterpart. We refer to \cite{CGQ22,Shi24,BC23,MS25,Kig23,CGYZ26,AB25,AES25a} for recent developments on nonlinear $p$-energies on fractals and metric measure spaces.

In our previous work \cite{Yan25d}, we proved that, under suitable geometric and analytic assumptions, the elliptic Harnack inequality and the two-sided capacity bounds imply the cutoff Sobolev inequality:
\begin{equation}\label{eq_EHIcapCS}
\text{EHI}+\text{cap}(\Psi)\Rightarrow\text{CS}(\Psi).
\end{equation}
We also proved the converse implication
\begin{equation}\label{eq_PICS}
\text{PI}(\Psi)+\text{CS}(\Psi)\Rightarrow\text{EHI}+\text{cap}(\Psi).
\end{equation}
Thus, the remaining step toward the full nonlinear analogue of the classical equivalence is to derive the Poincar\'e inequality from the elliptic Harnack inequality and the two-sided capacity bounds.

The main result of the present paper establishes precisely this implication. More specifically, for a $p$-energy on a volume doubling metric measure space, under the same structural assumptions as those used to prove (\ref{eq_EHIcapCS}) in \cite{Yan25d}, we prove that
$$\text{EHI}+\text{cap}(\Psi)\Rightarrow \text{PI}(\Psi).$$
Combining this implication with (\ref{eq_EHIcapCS}) and (\ref{eq_PICS}), we obtain the equivalence
\begin{equation}\label{eq_equiv}
\text{EHI}+\text{cap}(\Psi)\Leftrightarrow\text{PI}(\Psi)+\text{CS}(\Psi),
\end{equation}
under the same structural assumptions.

As a consequence, the nonlinear theory now has the same structural form as the classical Dirichlet form theory at the elliptic level: the Harnack--capacity side and the Poincar\'e--cutoff side are equivalent. This gives a robust substitute, in the absence of heat kernels, for the classical equivalences known in the linear setting.

The proof combines the technique developed in our previous work \cite{Yan25d} with a very recent technique introduced in \cite{EK26}. In \cite{EK26}, the authors established, for the first time in the \emph{nonlinear} setting, the equivalence between the parabolic Harnack inequality $\text{PHI}$ and the conjunction of the volume doubling condition and the Poincar\'e inequality $\text{PI}$, on metric measure spaces endowed with an upper gradient structure. Certain arguments in \cite{EK26}, however, are sufficiently robust to be adapted to more general metric measure spaces, including fractals. Following \cite[Lemmas 2.5 and 7.5, Theorem 7.6]{EK26}, the proof of $\text{PI}(\Psi)$ can be reduced to establishing $\text{EHI}$ together with a version of the Faber--Krahn inequality $\text{FK}(\Psi)$, which they call the Poincar\'e-type inequality for functions with zero boundary values; see \cite[(7.1), (7.10)]{EK26}. In the setting of \cite{EK26}, both $\text{FK}$ and $\text{EHI}$ are obtained from $\text{PHI}$: $\text{FK}$ follows from \cite[Proposition 7.1]{EK26}, while $\text{PHI}$ directly implies $\text{EHI}$. In our setting, since $\text{EHI}$ is assumed, it remains to establish $\text{FK}(\Psi)$. Building on our previous result in \cite{Yan25d} on Wolff potential estimates under $\text{EHI}$, $\text{cap}(\Psi)$, and certain geometric and analytic assumptions, we first prove an estimate of the measure of a set in terms of its relative capacity. Combining this estimate with a standard Maz'ya decomposition technique, we obtain $\text{FK}(\Psi)$, and hence $\text{PI}(\Psi)$.

We now turn to the formal statement of our result. To begin with, let us fix the notation. Throughout this paper, $p\in(1,+\infty)$ is fixed. The letters $C,C_1,C_2,C_A, C_B$ will always refer to some positive constants and may change at each occurrence. The sign $\asymp$ means that the ratio of the two sides is bounded from above and below by positive constants. The sign $\lesssim$ ($\gtrsim$) means that the LHS is bounded by positive constant times the RHS from above (below). For $x,y\in \mathbb{R}$, denote $x\vee y=\max\{x,y\}$, $x\wedge y=\min\{x,y\}$, $x_+=\max\{x,0\}$, and $x_-=\max\{-x,0\}$. We use $\# A$ to denote the cardinality of a set $A$.

We say that a function $\Psi:[0,+\infty)\to[0,+\infty)$ is doubling if $\Psi$ is a homeomorphism, which implies that $\Psi$ is strictly increasing continuous and $\Psi(0)=0$, and there exists $C_\Psi>1$, called a doubling constant of $\Psi$, such that $\Psi(2r)\le C_\Psi\Psi(r)$ for any $r>0$. Throughout this paper, we always assume that $\Psi$ is a doubling function with a doubling constant $C_\Psi$, and there exist $\beta_*, \beta^*>0$ with $\beta_*\le\beta^*$ such that
$$\frac{1}{C_\Psi}\left(\frac{R}{r}\right)^{\beta_*}\le \frac{\Psi(R)}{\Psi(r)}\le C_{\Psi}\left(\frac{R}{r}\right)^{\beta^*}\text{ for any }r\le R.$$
Indeed, we can take $\beta^*=\log_2C_\Psi$.

Let $(X,d,m)$ be a complete unbounded metric measure space, that is, $(X,d)$ is a complete unbounded locally compact separable metric space and $m$ is a positive Radon measure on $X$ with full support. Throughout this paper, we always assume that all metric balls are relatively compact. For any $x\in X$ and any $r>0$, denote $B(x,r)=\{y\in X:d(x,y)<r\}$ and $V(x,r)=m(B(x,r))$. If $B=B(x,r)$, then denote $\delta B=B(x,\delta r)$ for any $\delta>0$. Let $\mathcal{B}(X)$ be the family of all Borel measurable subsets of $X$. Let $C(X)$ be the family of all continuous functions on $X$. Let $C_c(X)$ be the family of all continuous functions on $X$ with compact support. Denote $\dashint_A=\frac{1}{m(A)}\int_A$ and $u_A=\dashint_Au\mathrm{d} m$ for any measurable set $A$ with $m(A)\in(0,+\infty)$ and any function $u$ such that the integral $\int_Au\mathrm{d} m$ is well-defined.

We say that the chain condition \ref{eq_CC} holds if there exists $C_{cc}>0$ such that for any $x,y\in X$ and any integer $n\ge1$, there exists a sequence $\{x_k:0\le k\le n\}$ of points in $X$ with $x_0=x$ and $x_n=y$ such that
\begin{equation*}\label{eq_CC}\tag*{CC}
d(x_k,x_{k-1})\le C_{cc} \frac{d(x,y)}{n}\text{ for any }k=1,\ldots,n.
\end{equation*}
Throughout this paper, we always assume \ref{eq_CC}.

We say that the linearly locally connected condition \hypertarget{eq_LLC}{\text{LLC}} holds if there exists $A_{LLC}>1$ such that for any $x_0\in X$ and any $r>0$, for any $x,y\in B(x_0,r)\backslash B(x_0,r/2)$, there exists a connected compact set $K\subseteq B(x_0,A_{LLC}r)\backslash B(x_0,r/(2A_{LLC}))$ containing $x,y$. See also \cite[3.12]{HK98} or \cite[Page 234]{HKST15}.

We say that the volume doubling condition \ref{eq_VD} holds if there exists $C_{VD}>0$ such that
\begin{equation*}\label{eq_VD}\tag*{VD}
V(x,2r)\le C_{VD}V(x,r)\text{ for any }x\in X,r>0.
\end{equation*}
It follows directly from \ref{eq_VD} that, for any $x,y\in X$ and any $R,r>0$ with $r\le R$, we have
\begin{equation}\label{eq_VDa}
\frac{V(x,R)}{V(y,r)}\le C_{VD} \left(\frac{d(x,y)+R}{r}\right)^{\alpha},
\end{equation}
where $\alpha=\log_2C_{VD}$.

We say that $(\mathcal{E},\mathcal{F})$ is a $p$-energy on $(X,d,m)$ if $\mathcal{F}$ is a dense subspace of $L^p(X;m)$ and $\mathcal{E}:\mathcal{F}\to[0,+\infty)$ satisfies the following conditions.

\begin{enumerate}[label=(\arabic*)]
\item $\mathcal{E}^{1/p}$ is a semi-norm on $\mathcal{F}$, that is, for any $f,g\in\mathcal{F}$, $c\in\mathbb{R}$, we have $\mathcal{E}(f)\ge0$, $\mathcal{E}(cf)^{1/p}=|c|\mathcal{E}(f)^{1/p}$ and $\mathcal{E}(f+g)^{1/p}\le\mathcal{E}(f)^{1/p}+\mathcal{E}(g)^{1/p}$.
\item (Closed property) $(\mathcal{F},\mathcal{E}(\cdot)^{1/p}+\lVert {\cdot}\rVert_{L^p(X;m)})$ is a Banach space.
\item (Markovian property) For any $\varphi\in C(\mathbb{R})$ with $\varphi(0)=0$ and $|\varphi(t)-\varphi(s)|\le|t-s|$ for any $t,s\in\mathbb{R}$, for any $f\in\mathcal{F}$, we have $\varphi(f)\in\mathcal{F}$ and $\mathcal{E}(\varphi(f))\le\mathcal{E}(f)$.
\item (Regular property) $\mathcal{F}\cap C_c(X)$ is uniformly dense in $C_c(X)$ and $(\mathcal{E}(\cdot)^{1/p}+\lVert {\cdot}\rVert_{L^p(X;m)})$-dense in $\mathcal{F}$.
\item (Strongly local property) For any $f,g\in\mathcal{F}$ with compact support and $g$ constant in an open neighborhood of $\mathrm{supp}(f)$, we have $\mathcal{E}(f+g)=\mathcal{E}(f)+\mathcal{E}(g)$.
\item ($p$-Clarkson's inequality) For any $f,g\in\mathcal{F}$, we have
\begin{equation*}\label{eq_Cla}\tag*{Cla}
\begin{cases}
\mathcal{E}(f+g)+\mathcal{E}(f-g)\ge2 \left(\mathcal{E}(f)^{\frac{1}{p-1}}+\mathcal{E}(g)^{\frac{1}{p-1}}\right)^{p-1}&\text{if }p\in(1,2],\\
\mathcal{E}(f+g)+\mathcal{E}(f-g)\le2 \left(\mathcal{E}(f)^{\frac{1}{p-1}}+\mathcal{E}(g)^{\frac{1}{p-1}}\right)^{p-1}&\text{if }p\in[2,+\infty).\\
\end{cases}
\end{equation*}
\end{enumerate}
Moreover, we also always assume the following condition.
\begin{itemize}
\item ($\mathcal{F}\cap L^\infty(X;m)$ is an algebra) For any $f,g\in\mathcal{F}\cap L^\infty(X;m)$, we have $fg\in\mathcal{F}$ and
\begin{equation*}\label{eq_Alg}\tag*{Alg}
\mathcal{E}(fg)^{{1}/{p}}\le \lVert f\rVert_{L^\infty(X;m)}\mathcal{E}(g)^{1/p}+\lVert g\rVert_{L^\infty(X;m)}\mathcal{E}(f)^{1/p}.
\end{equation*}
\end{itemize}
Denote $\mathcal{E}_\lambda(\cdot)=\mathcal{E}(\cdot)+\lambda \lVert {\cdot}\rVert^p_{L^p(X;m)}$ for any $\lambda>0$. Indeed, a general condition called the generalized $p$-contraction property was introduced in \cite{KS24a}, which implies \ref{eq_Cla}, \ref{eq_Alg}, and holds on a large family of metric measure spaces.

We list some basic properties of $p$-energies as follows. For any $f,g\in\mathcal{F}$, the derivative
$$\mathcal{E}(f;g)=\frac{1}{p}\frac{\mathrm{d}}{\mathrm{d} t}\mathcal{E}(f+tg)|_{t=0}\in\mathbb{R}$$
exists, the map $\mathcal{E}(f;\cdot):\mathcal{F}\to\mathbb{R}$ is linear, $\mathcal{E}(f;f)=\mathcal{E}(f)$. Moreover, for any $f,g\in\mathcal{F}$, for any $a\in\mathbb{R}$, we have
\begin{equation}\label{eq_quasi_strict}
\mathbb{R}\ni t\mapsto\mathcal{E}(f+tg;g)\in\mathbb{R}\text{ is strictly increasing if and only if }\mathcal{E}(g)>0,
\end{equation}
$$\mathcal{E}(af;g)=\mathrm{sgn}(a)|a|^{p-1}\mathcal{E}(f;g),$$
$$|\mathcal{E}(f;g)|\le\mathcal{E}(f)^{(p-1)/p}\mathcal{E}(g)^{1/p}.$$
Moreover, for any $\lambda>0$, all of the above results remain valid with $\mathcal{E}$ replaced by $\mathcal{E}_\lambda$, and for any $f,g\in\mathcal{F}$, we have
$$\mathcal{E}_\lambda(f;g)=\mathcal{E}(f;g)+\lambda\int_X\mathrm{sgn}(f)|f|^{p-1}g\mathrm{d} m.$$
See \cite[Theorem 3.7, Corollary 3.25]{KS24a} for the proofs of these results.

By \cite[Theorem 2.4]{Sas26}, a $p$-energy $(\mathcal{E},\mathcal{F})$ corresponds to a (canonical) $p$-energy measure $\Gamma:\mathcal{F}\times\mathcal{B}(X)\to[0,+\infty)$, $(f,A)\mapsto\Gamma(f)(A)$ satisfying the following conditions.
\begin{enumerate}[label=(\alph*),ref=(\alph*)]
\item\label{item_meas1} For any $f\in\mathcal{F}$, $\Gamma(f)(\cdot)$ is a positive Radon measure on $X$ with $\Gamma(f)(X)=\mathcal{E}(f)$.
\item\label{item_meas2} For any $A\in\mathcal{B}(X)$, $\Gamma(\cdot)(A)^{1/p}$ is a semi-norm on $\mathcal{F}$.
\item\label{item_meas3} For any $f,g\in\mathcal{F}\cap C_c(X)$, $A\in\mathcal{B}(X)$, if $f-g$ is constant on $A$, then $\Gamma(f)(A)=\Gamma(g)(A)$.
\item\label{item_meas4} ($p$-Clarkson's inequality) For any $f,g\in\mathcal{F}$, for any $A\in\mathcal{B}(X)$, we have
\begin{equation*}
\begin{cases}
\Gamma(f+g)(A)+\Gamma(f-g)(A)\ge2 \left(\Gamma(f)(A)^{\frac{1}{p-1}}+\Gamma(g)(A)^{\frac{1}{p-1}}\right)^{p-1}&\text{if }p\in(1,2],\\
\Gamma(f+g)(A)+\Gamma(f-g)(A)\le2 \left(\Gamma(f)(A)^{\frac{1}{p-1}}+\Gamma(g)(A)^{\frac{1}{p-1}}\right)^{p-1}&\text{if }p\in[2,+\infty).\\
\end{cases}
\end{equation*}
\item\label{item_meas5} (Chain rule) For any $f,g\in\mathcal{F}\cap C_c(X)$, for any piecewise $C^1$ functions $\varphi,\psi:\mathbb{R}\to\mathbb{R}$ with $\varphi(0)=\psi(0)=0$, we have
$$\mathrm{d}\Gamma\left(\varphi(f);\psi(g)\right)=\mathrm{sgn}(\varphi'(f))|\varphi'(f)|^{p-1}\psi'(g)\mathrm{d}\Gamma(f;g),$$
where $\Gamma(f;g)$ is a signed measure given by $\Gamma(f;g)=\frac{1}{p}\frac{\mathrm{d}}{\mathrm{d}t}\Gamma(f+tg)|_{t=0}$.
\item\label{item_meas6} (Leibniz rule) For any $f,g,h\in \mathcal{F}\cap C_c(X)$, we have $\mathrm{d}\Gamma(f;gh)=g \mathrm{d}\Gamma(f;h)+h \mathrm{d}\Gamma(f;g)$.
\end{enumerate}
%Using the chain rule, we have the following condition.
%\begin{itemize}
%\item (Strong sub-additivity) For any $f,g\in\mathcal{F}$, we have $f\vee g, f\wedge g\in\mathcal{F}$ and
%\begin{equation*}\label{eq_SubAdd}\tag*{SubAdd}
%\mathcal{E}(f\vee g)+\mathcal{E}(f\wedge g)\le\mathcal{E}(f)+\mathcal{E}(g).
%\end{equation*}
%\end{itemize}

Let $A_1,A_2\in\mathcal{B}(X)$. We define the capacity between $A_1,A_2$ as
\begin{align*}
&\mathrm{cap}(A_1,A_2)=\inf\left\{\mathcal{E}(\varphi):\varphi\in\mathcal{F},
\begin{array}{l}
\varphi=1\text{ in an open neighborhood of }A_1,\\
\varphi=0\text{ in an open neighborhood of }A_2
\end{array}
\right\},
\end{align*}
here we use the convention that $\inf\emptyset=+\infty$.

We say that the two-sided capacity bounds \hypertarget{eq_cap}{$\text{cap}(\Psi)$} hold if both the capacity upper bound \ref{eq_ucap} and the capacity lower bound \ref{eq_lcap} hold as follows. There exist $C_{cap}>0$, $A_{cap}>1$ such that for any ball $B(x,r)$, we have
\begin{align*}
\mathrm{cap}\left(B(x,r),X\backslash B(x,A_{cap}r)\right)&\le C_{cap} \frac{V(x,r)}{\Psi(r)},\label{eq_ucap}\tag*{$\text{cap}(\Psi)_{\le}$}\\
\mathrm{cap}\left(B(x,r),X\backslash B(x,A_{cap}r)\right)&\ge \frac{1}{C_{cap}} \frac{V(x,r)}{\Psi(r)}.\label{eq_lcap}\tag*{$\text{cap}(\Psi)_{\ge}$}
\end{align*}

Let $U$ be an open subset of $X$. Let
$$\mathcal{F}(U)=\text{the }\mathcal{E}_1\text{-closure of }\mathcal{F}\cap C_c(U).$$
We say that $u\in\mathcal{F}$ is harmonic in $U$ if $\mathcal{E}(u;v)=0$ for any $v\in\mathcal{F}\cap C_c(U)$, denoted by $-\Delta_pu=0$ in $U$. We say that $u\in\mathcal{F}$ is superharmonic in $U$ (resp. subharmonic in $U$) if $\mathcal{E}(u;v)\ge0$ (resp. $\mathcal{E}(u;v)\le0$) for any non-negative $v\in\mathcal{F}\cap C_c(U)$, denoted by $-\Delta_pu\ge0$ in $U$ (resp. $-\Delta_pu\le0$ in $U$). Moreover, if $U$ is bounded, then by \cite[Lemma 6.2]{Yan25c}, we have
$$\mathcal{F}(U)=\left\{u\in\mathcal{F}:\widetilde{u}=0\text{ q.e. on }X\backslash U\right\}.$$
Assuming that $U$ is bounded, then the above equality for harmonic functions, as well as the corresponding inequalities for superharmonic and subharmonic functions, also hold for all $v\in\mathcal{F}(U)$.

We say that the bottom spectrum positivity condition \ref{eq_BSP} holds if for any bounded open subset $U$ of $X$, we have
\begin{equation*}\label{eq_BSP}\tag*{\text{BSP}}
\lambda_1(U)=\inf \left\{\frac{\mathcal{E}(u)}{\lVert u\rVert_{L^p(X;m)}^p}:u\in \mathcal{F}(U)\backslash \{0\}\right\}>0.
\end{equation*}

We say that the lower semi-continuous condition \hypertarget{eq_LSC}{\text{LSC}} holds if for any bounded open subset $U$ of $X$, for any $u\in \mathcal{F}$ which is bounded from below superharmonic in $U$, we have $u$ has a modification which is lower semi-continuous in $U$.

We say that the elliptic Harnack inequality \ref{eq_EHI} holds if there exist $C_H>0$, $A_H>1$ such that for any ball $B(x,r)$, for any $u\in\mathcal{F}$ which is non-negative harmonic in $B(x,A_Hr)$, we have
\begin{equation*}\label{eq_EHI}\tag*{\text{EHI}}
\esup_{B(x,r)}u\le C_H\einf_{B(x,r)}u.
\end{equation*}

We say that the annulus elliptic Harnack inequality \ref{eq_EHIann} holds if there exist $A_1$, $A_2$, $A_3>1$ with $A_1<A_2<A_3$ and $C_H>0$ such that the following holds: for any ball $B(x,r)$ and any $u\in \mathcal{F}$ which is non-negative superharmonic in $B(x,A_3r)$ and harmonic in $B(x,A_3r)\backslash \overline{B(x,r)}$, we have
\begin{equation*}\label{eq_EHIann}\tag*{$\text{EHI}_{\text{ann}}$}
\esup_{B(x,A_2r)\backslash B(x,A_1r)}u\le C_H\einf_{B(x,A_2r)\backslash B(x,A_1r)}u.
\end{equation*}
See also \cite[Theorem 3.1]{Bla01} and \cite[LEMMA 6.3]{GS05} for similar results. However, in addition to harmonicity in the annulus, we also require superharmonicity in the entire ball. By a standard chaining argument, as in these references, it is easy to see that
\begin{center}
\ref{eq_VD} + \hyperlink{eq_LLC}{\text{LLC}} + \ref{eq_EHI} $\Rightarrow$ \ref{eq_EHIann}.
\end{center}

We say that the Poincar\'e inequality \ref{eq_PI} holds if there exist $C_{PI}>0$, $A_{PI}\ge1$ such that for any ball $B(x,r)$, for any $f\in\mathcal{F}$, we have
\begin{equation*}\label{eq_PI}\tag*{PI($\Psi$)}
\int_{B(x,r)}\lvert f-f_{B(x,r)}\rvert^p\mathrm{d} m\le C_{PI}\Psi(r)\int_{{B(x,A_{PI}r)}}\mathrm{d}\Gamma(f).
\end{equation*}
Under \ref{eq_VD} and \ref{eq_PI}, both \ref{eq_BSP} and \ref{eq_lcap} hold; these follow from \cite[Lemma 4.1]{Yan25c} and \cite[Lemma 5.1 (a)]{BB04}, respectively.

The main result of this paper is as follows.

\begin{theorem}\label{thm_main}
Assume \ref{eq_VD}, \hyperlink{eq_LSC}{\text{LSC}}, \ref{eq_BSP}, \ref{eq_EHI}, \hyperlink{eq_cap}{$\text{cap}(\Psi)$}. If either
\begin{enumerate}[label=(\alph*),ref=(\alph*)]
\item\label{cond_main_LLC} \hyperlink{eq_LLC}{\text{LLC}}, or
\item\label{cond_main_EHIann} \ref{eq_EHIann},
\end{enumerate}
holds, then \ref{eq_PI} holds.
\end{theorem}

\begin{remark}
\begin{enumerate}[label=(\arabic*)]
\item Since \ref{eq_VD}, \hyperlink{eq_LLC}{\text{LLC}}, \ref{eq_EHI} imply \ref{eq_EHIann}, it suffices here to treat case \ref{cond_main_EHIann}.
\item For the case $p=2$, it is well-known in the theory of heat kernel estimates that, under \ref{eq_VD}, the conjunction of \ref{eq_EHI} and \hyperlink{eq_cap}{$\text{cap}(\Psi)$} already implies \ref{eq_PI}; see \cite[THEOREM 1.2]{GHL15}. However, for general $p>1$, due to the \emph{nonlinearity} and the limitation of our techniques, \hyperlink{eq_LSC}{\text{LSC}}, \ref{eq_BSP}, and either \hyperlink{eq_LLC}{\text{LLC}} or \ref{eq_EHIann} are additionally required.
\end{enumerate}
\end{remark}

Combining Theorem \ref{thm_main} with our earlier results in \cite{Yan25d}, we obtain several versions of (\ref{eq_equiv}). To state these results, we first introduce the cutoff Sobolev inequality.

Let $U,V$ be two open subsets of $X$ satisfying $U\subseteq\overline{U}\subseteq V$. We say that $\phi\in\mathcal{F}$ is a cutoff function for $U\subseteq V$ if $0\le\phi\le1$ in $X$, $\phi=1$ in an open neighborhood of $\overline{U}$ and $\mathrm{supp}(\phi)\subseteq V$, where $\mathrm{supp}(f)$ refers to the support of the measure of $|f|\mathrm{d} m$ for any given function $f$.

We say that the cutoff Sobolev inequality \ref{eq_CS} holds if there exist $C_{1},C_{2}>0$, $A_{S}>1$ such that for any ball $B(x,r)$, there exists a cutoff function $\phi\in\mathcal{F}$ for $B(x,r)\subseteq B(x,A_Sr)$ such that for any $f\in\mathcal{F}$, we have
\begin{equation*}\label{eq_CS}\tag*{CS($\Psi$)}
\int_{B(x,A_{S}r)}|\widetilde{f}|^p\mathrm{d}\Gamma(\phi)\le C_{1}\int_{B(x,A_{S}r)}\mathrm{d}\Gamma(f)+\frac{C_{2}}{\Psi(r)}\int_{B(x,A_{S}r)}|f|^p\mathrm{d} m,
\end{equation*}
where $\widetilde{f}$ is a quasi-continuous modification of $f$, such that $\widetilde{f}$ is uniquely determined $\Gamma(\phi)$-a.e. in $X$, see \cite[Section 8]{Yan25a} for more details.

%Under \ref{eq_VD}, by taking $f\equiv1$ in $B(x,A_Sr)$, it is easy to see that \ref{eq_CS} implies \ref{eq_ucap}.

We have the following equivalences.

\begin{corollary}\label{cor_EHIann}
Assume \ref{eq_VD}. The followings are equivalent.
\begin{enumerate}[label=(\arabic*),ref=(\arabic*)]
\item\label{item_EHIann1} \hyperlink{eq_LSC}{\text{LSC}}, \ref{eq_BSP}, \ref{eq_EHI}, \ref{eq_EHIann}, \hyperlink{eq_cap}{$\text{cap}(\Psi)$}.
\item\label{item_EHIann2} \ref{eq_PI}, \ref{eq_CS}.
\end{enumerate}
\end{corollary}

\begin{corollary}\label{cor_LLC}
Assume \ref{eq_VD}, \hyperlink{eq_LLC}{\text{LLC}}. The followings are equivalent.
\begin{enumerate}[label=(\roman*),ref=(\roman*)]
\item\label{item_LLC1} \hyperlink{eq_LSC}{\text{LSC}}, \ref{eq_BSP}, \ref{eq_EHI}, \hyperlink{eq_cap}{$\text{cap}(\Psi)$}.
\item\label{item_LLC2} \ref{eq_PI}, \ref{eq_CS}.
\end{enumerate}
\end{corollary}

\begin{proof}[Proofs of Corollaries \ref{cor_EHIann} and \ref{cor_LLC}]
The implications ``\ref{item_EHIann2}$\Rightarrow$\ref{item_EHIann1}" and ``\ref{item_LLC2}$\Rightarrow$\ref{item_LLC1}" follow directly from \cite[Theorem 2.3]{Yan25d}. For the converse implications, \ref{eq_PI} follows from Theorem \ref{thm_main}, while \ref{eq_CS} follows from \cite[Theorem 2.1]{Yan25d}.
\end{proof}

\section{Proof}\label{sec_proof}

We will frequently use the following elementary consequence of H\"older's inequality:
\begin{equation}\label{eq_ele}
\int_{B(x_0,R)}\lvert f-f_{B(x_0,R)}\rvert^p \mathrm{d}m\le 2^p\inf_{c\in \mathbb{R}}\int_{B(x_0,R)}\lvert f-c\rvert^p \mathrm{d}m,
\end{equation}
for any ball $B(x_0,R)$ and any function $f$ for which the integrals above are well-defined.

We divide the proof into two steps. In the first step, following the strategy of \cite{EK26}, we reduce the proof of \ref{eq_PI} to proving the following Faber--Krahn inequality.

We say that the Faber--Krahn inequality \ref{eq_FK} holds if there exists $C>0$ such that for any ball $B(x_0,R)$ and any $f\in \mathcal{F}(B(x_0,R))$, we have
\begin{equation*}\label{eq_FK}\tag*{$\text{FK}(\Psi)$}
\int_{B(x_0,R)}\lvert f\rvert^p \mathrm{d}m\le C\Psi(R)\int_{B(x_0,R)} \mathrm{d}\Gamma(f).
\end{equation*}
Obviously, \ref{eq_FK} implies \ref{eq_BSP}.

Let
\begin{align*}
&\mathcal{F}_{\mathrm{loc}}=\left\{u:
\begin{array}{l}
\text{for any relatively compact open set }U,\\
\text{there exists }u^\#\in \mathcal{F}\text{ such that }u=u^\#\text{ }m\text{-a.e. in }U
\end{array}
\right\}.
\end{align*}
For any $u\in \mathcal{F}_{\mathrm{loc}}$, let $\Gamma(u)|_U=\Gamma(u^\#)|_U$, where $u^\#, U$ are given as above, then $\Gamma(u)$ is a well-defined positive Radon measure on $X$, as follows from \ref{item_meas3} and \ref{item_meas6} in the definition of $\Gamma(u)$ for $u\in \mathcal{F}$, together with an argument similar to that in \cite[Corollary 3.2.1]{FOT11}.

\begin{proposition}\label{prop_PI}
Assume that there exist $C>0$, $A>1$, $\varepsilon>0$ with $\frac{1+\varepsilon}{2^{\frac{p-1}{p}}}<1$ such that for any ball $B(x_0,R)$ and any non-negative $f\in \mathcal{F}_{\mathrm{loc}}$, we have
\begin{equation}\label{eq_weakPI}
\left(\dashint_{B(x_0,R)}f^p \mathrm{d}m\right)^{1/p}\le C\left(\frac{\Psi(R)}{V(x_0,R)}\int_{B(x_0,AR)} \mathrm{d}\Gamma(f)\right)^{1/p}+(1+\varepsilon) \dashint_{B(x_0,R)}f \mathrm{d}m.
\end{equation}
Then \ref{eq_PI} holds.
\end{proposition}

\begin{proof}
For any ball $B(x_0,R)$ and any $f\in \mathcal{F}$, there exists $c\in \mathbb{R}$ such that
\begin{equation}\label{eq_medium}
\begin{aligned}
m \left(\left\{f\ge c\right\}\cap B(x_0,R)\right)&\ge \frac{1}{2}V(x_0,R),\\
m \left(\left\{f\le c\right\}\cap B(x_0,R)\right)&\ge \frac{1}{2}V(x_0,R).
\end{aligned}
\end{equation}
Applying (\ref{eq_weakPI}) to $(f-c)_+\in \mathcal{F}_{\mathrm{loc}}$, we have
\begin{align*}
&\left(\dashint_{B(x_0,R)}(f-c)_+^p \mathrm{d}m\right)^{1/p}\\
&\le C\left(\frac{\Psi(R)}{V(x_0,R)}\int_{B(x_0,AR)} \mathrm{d}\Gamma((f-c)_+)\right)^{1/p}+(1+\varepsilon) \dashint_{B(x_0,R)}(f-c)_+ \mathrm{d}m.
\end{align*}
By the Markovian property, we have
$$\int_{B(x_0,AR)} \mathrm{d}\Gamma((f-c)_+)\le \int_{B(x_0,AR)} \mathrm{d}\Gamma(f).$$
By H\"older's inequality and (\ref{eq_medium}), we have
\begin{align*}
&\dashint_{B(x_0,R)}(f-c)_+ \mathrm{d}m\\
&\le \left(\frac{m \left(\left\{f>c\right\}\cap B(x_0,R)\right)}{V(x_0,R)}\right)^{1-\frac{1}{p}} \left(\dashint_{B(x_0,R)}(f-c)_+^p \mathrm{d}m\right)^{\frac{1}{p}}\\
&\le \frac{1}{2^{\frac{p-1}{p}}}\left(\dashint_{B(x_0,R)}(f-c)_+^p \mathrm{d}m\right)^{\frac{1}{p}}.
\end{align*}
Hence
\begin{align*}
&\left(\dashint_{B(x_0,R)}(f-c)_+^p \mathrm{d}m\right)^{1/p}\\
&\le C\left(\frac{\Psi(R)}{V(x_0,R)}\int_{B(x_0,AR)} \mathrm{d}\Gamma(f)\right)^{1/p}+\frac{1+\varepsilon}{2^{\frac{p-1}{p}}}\left(\dashint_{B(x_0,R)}(f-c)_+^p \mathrm{d}m\right)^{{1}/{p}}.
\end{align*}
Let $\gamma=1-\frac{1+\varepsilon}{2^{\frac{p-1}{p}}}$, then by assumption, we have $\gamma\in(0,1)$, which gives
$$\int_{B(x_0,R)}(f-c)_+^p \mathrm{d}m\le \left(\frac{C}{\gamma}\right)^p{\Psi(R)}\int_{B(x_0,AR)} \mathrm{d}\Gamma(f).$$
Similarly,
$$\int_{B(x_0,R)}(f-c)_-^p \mathrm{d}m\le \left(\frac{C}{\gamma}\right)^p{\Psi(R)}\int_{B(x_0,AR)} \mathrm{d}\Gamma(f).$$
Therefore, by (\ref{eq_ele}),
\begin{align*}
&\int_{B(x_0,R)} \lvert f-f_{B(x_0,R)}\rvert^p \mathrm{d}m\le 2^p\int_{B(x_0,R)} \lvert f-c\rvert^p \mathrm{d}m\\
&=2^p\int_{B(x_0,R)} (f-c)_+^p \mathrm{d}m+2^p\int_{B(x_0,R)} (f-c)_-^p \mathrm{d}m\\
&\le 2^{p+1}\left(\frac{C}{\gamma}\right)^p{\Psi(R)}\int_{B(x_0,AR)} \mathrm{d}\Gamma(f).
\end{align*}
\end{proof}

The elliptic Harnack inequality \ref{eq_EHI} implies the following standard oscillation estimate.

\begin{lemma}\label{lem_osc}
Assume \ref{eq_EHI}. Then there exist $C>0$, $\kappa\in(0,1]$ depending only on $C_H,A_H$, such that for any ball $B(x_0,R)$, for any $u\in \mathcal{F}_{\mathrm{loc}}$ which is harmonic in $B(x_0,R)$, for any $r\in(0,R)$, we have
$$\eosc_{B(x_0,r)}u \le C \left(\frac{r}{R}\right)^\kappa \eosc_{B(x_0,R)}u.$$
\end{lemma}

\begin{proposition}\label{prop_weakPI}
Assume \ref{eq_VD}, \ref{eq_FK}, \ref{eq_EHI}. Then (\ref{eq_weakPI}) holds, and hence \ref{eq_PI} holds.
\end{proposition}

\begin{proof}
Let $A=A_H>1$, with $A_H$ as in \ref{eq_EHI}. Firstly, we prove that there exists $C_1>0$ such that for any ball $B(x_0,R)$, for any non-negative $f\in \mathcal{F}_{\mathrm{loc}}$, for any $\delta\in(0,1)$, we have
\begin{equation}\label{eq_weakPI1}
\int_{B(x_0,R)}f^p \mathrm{d}m\le \frac{C_1\Psi(R)}{\delta^\alpha}\int_{B(x_0,AR)}\mathrm{d}\Gamma(f)+C_1V(x_0,R)\left(\dashint_{B(x_0,\delta R)}f \mathrm{d}m\right)^p,
\end{equation}
where $\alpha$ is the constant appearing in (\ref{eq_VDa}).

Without loss of generality, we may assume that $f$ is bounded. By \cite[Proposition 4.2]{Yan25c} and its proof, together with the regular and strongly local properties, there exists a non-negative bounded $u\in \mathcal{F}_{\mathrm{loc}}$ such that $u$ is harmonic in $B(x_0,AR)$ and $\widetilde{u}=\widetilde{f}$ q.e. on $X\backslash B(x_0,AR)$, which gives $f-u\in \mathcal{F}(B(x_0,AR))$. Moreover, we have
\begin{equation}\label{eq_weakPI2}
\int_{B(x_0,AR)}\mathrm{d}\Gamma(u)=\inf \left\{\int_{B(x_0,AR)}\mathrm{d}\Gamma(v):\widetilde{v}=\widetilde{f}\text{ q.e. on }X\backslash B(x_0,AR)\right\}\le \int_{B(x_0,AR)}\mathrm{d}\Gamma(f).
\end{equation}

Notice that
\begin{align}
\left(\dashint_{B(x_0,R)}f^p \mathrm{d}m\right)^{1/p}\le \left(\dashint_{B(x_0,R)} u^p \mathrm{d}m\right)^{1/p}+\left(\dashint_{B(x_0,R)}\lvert f-u\rvert^p \mathrm{d}m\right)^{1/p}.\label{eq_weakPI3}
\end{align}
By \ref{eq_EHI}, we have
\begin{align}
&\left(\dashint_{B(x_0,R)} u^p \mathrm{d}m\right)^{1/p}\le \esup_{B(x_0,R)} u\le C_H \einf_{B(x_0,R)}u\le C_H \einf_{B(x_0,\delta R)}u\nonumber\\
&\le C_H \dashint_{B(x_0,\delta R)}u \mathrm{d}m\le C_H\dashint_{B(x_0,\delta R)}\lvert f-u\rvert \mathrm{d}m+C_H\dashint_{B(x_0,\delta R)}f \mathrm{d}m.\label{eq_weakPI4}
\end{align}
Since $f-u\in \mathcal{F}(B(x_0,AR))$, by H\"older's inequality, (\ref{eq_VDa}), \ref{eq_FK}, and (\ref{eq_weakPI2}), we have
\begin{align}
&\dashint_{B(x_0,\delta R)}\lvert f-u\rvert \mathrm{d}m\le \left(\dashint_{B(x_0,\delta R)}\lvert f-u\rvert^p \mathrm{d}m\right)^{1/p}\nonumber\\
&\le\left(\frac{1}{V(x_0,\delta R)}\int_{B(x_0,A R)}\lvert f-u\rvert^p \mathrm{d}m\right)^{1/p}\lesssim \left(\frac{\Psi(R)}{\delta^{\alpha}V(x_0,R)}\int_{B(x_0,AR)}\mathrm{d}\Gamma(f-u)\right)^{1/p}\nonumber\\
&\le \left(\frac{\Psi(R)}{\delta^{\alpha}V(x_0,R)}\int_{B(x_0,AR)}\mathrm{d}\Gamma(f)\right)^{1/p}+\left(\frac{\Psi(R)}{\delta^{\alpha}V(x_0,R)}\int_{B(x_0,AR)}\mathrm{d}\Gamma(u)\right)^{1/p}\nonumber\\
&\le2\left(\frac{\Psi(R)}{\delta^{\alpha}V(x_0,R)}\int_{B(x_0,AR)}\mathrm{d}\Gamma(f)\right)^{1/p}.\label{eq_weakPI5}
\end{align}
Similarly, we have
\begin{equation}\label{eq_weakPI6}
\left(\dashint_{B(x_0,R)}\lvert f-u\rvert^p \mathrm{d}m\right)^{1/p}\lesssim \left(\frac{\Psi(R)}{V(x_0,R)}\int_{B(x_0,AR)}\mathrm{d}\Gamma(f)\right)^{1/p}.
\end{equation}
Combining (\ref{eq_weakPI3})--(\ref{eq_weakPI6}), we have
\begin{equation*}
\left(\dashint_{B(x_0,R)}f^p \mathrm{d}m\right)^{1/p}\lesssim \left(\frac{\Psi(R)}{\delta^{\alpha}V(x_0,R)}\int_{B(x_0,AR)}\mathrm{d}\Gamma(f)\right)^{1/p}+\dashint_{B(x_0,\delta R)}f \mathrm{d}m.
\end{equation*}
Taking the $p$-th power of both sides, we obtain (\ref{eq_weakPI1}).

Secondly, we prove (\ref{eq_weakPI}). Without loss of generality, we may further assume that $f$ is bounded. Let $u\in \mathcal{F}_{\mathrm{loc}}$ be the function given in the first part of the proof. Let $\delta\in(0,1)$ be chosen later. Let $c=\einf_{B(x_0,\delta R)}u\in[0,+\infty)$. Applying (\ref{eq_weakPI1}) to $\lvert f-c\rvert\in \mathcal{F}_{\mathrm{loc}}$, we have
\begin{equation}\label{eq_weakPI7}
\int_{B(x_0,R)}\lvert f-c\rvert^p \mathrm{d}m\lesssim \frac{\Psi(R)}{\delta^\alpha}\int_{B(x_0,AR)}\mathrm{d}\Gamma(\lvert f-c\rvert)+V(x_0,R)\left(\dashint_{B(x_0,\delta R)}\lvert f-c\rvert \mathrm{d}m\right)^p.
\end{equation}
By the Markovian property, we have
\begin{equation}\label{eq_weakPI8}
\int_{B(x_0,AR)}\mathrm{d}\Gamma(\lvert f-c\rvert)\le \int_{B(x_0,AR)}\mathrm{d}\Gamma(f).
\end{equation}
Moreover,
\begin{equation*}
\dashint_{B(x_0,\delta R)}\lvert f-c\rvert \mathrm{d}m\le \dashint_{B(x_0,\delta R)}\lvert f-u\rvert \mathrm{d}m+\dashint_{B(x_0,\delta R)}(u-c)\mathrm{d}m,
\end{equation*}
where by (\ref{eq_weakPI5}), we have
\begin{equation*}
\dashint_{B(x_0,\delta R)}\lvert f-u\rvert \mathrm{d}m\lesssim \left(\frac{\Psi(R)}{\delta^{\alpha}V(x_0,R)}\int_{B(x_0,AR)}\mathrm{d}\Gamma(f)\right)^{1/p},
\end{equation*}
and by Lemma \ref{lem_osc} and \ref{eq_EHI}, we have
\begin{align*}
&\dashint_{B(x_0,\delta R)}(u-c)\mathrm{d}m\le \eosc_{B(x_0,\delta R)}u\lesssim \delta^\kappa \eosc_{B(x_0,R)}u\le\delta^\kappa \esup_{B(x_0,R)}u\le C_H \delta^\kappa \einf_{B(x_0,R)}u\\
&\le C_H\delta^\kappa\dashint_{B(x_0,R)}u \mathrm{d}m\le C_H\delta^\kappa \dashint_{B(x_0,R)}\lvert f-u\rvert \mathrm{d}m+C_H\delta^\kappa \dashint_{B(x_0,R)}f \mathrm{d}m\\
&\lesssim  \left(\frac{\Psi(R)}{V(x_0,R)}\int_{B(x_0,AR)}\mathrm{d}\Gamma(f)\right)^{1/p}+\delta^\kappa\dashint_{B(x_0,R)}f \mathrm{d}m,
\end{align*}
where the last inequality follows from H\"older's inequality and (\ref{eq_weakPI6}). Hence
\begin{align*}
\dashint_{B(x_0,\delta R)}\lvert f-c\rvert \mathrm{d}m\lesssim\left(\frac{\Psi(R)}{\delta^{\alpha}V(x_0,R)}\int_{B(x_0,AR)}\mathrm{d}\Gamma(f)\right)^{1/p}+\delta^\kappa\dashint_{B(x_0,R)}f \mathrm{d}m.
\end{align*}
Combining this with (\ref{eq_ele}), (\ref{eq_weakPI7}), and (\ref{eq_weakPI8}), we have
\begin{align*}
&\int_{B(x_0,R)} \lvert f-f_{B(x_0,R)}\rvert^p \mathrm{d}m\le 2^p \int_{B(x_0,R)} \lvert f-c\rvert^p \mathrm{d}m\\
&\lesssim \frac{\Psi(R)}{\delta^\alpha}\int_{B(x_0,AR)}\mathrm{d}\Gamma(f)+\delta^{p\kappa}V(x_0,R)\left(\dashint_{B(x_0,R)}f \mathrm{d}m\right)^p.
\end{align*}
Taking the $p$-th root of both sides, we obtain that there exists $C_2>0$ such that
\begin{align*}
&\left(\dashint_{B(x_0,R)} \lvert f-f_{B(x_0,R)}\rvert^p \mathrm{d}m\right)^{1/p}\\
&\le C_2 \left(\frac{\Psi(R)}{\delta^\alpha V(x_0,R)}\int_{B(x_0,AR)}\mathrm{d}\Gamma(f)\right)^{1/p}+C_2\delta^\kappa \dashint_{B(x_0,R)}f \mathrm{d}m.
\end{align*}
Therefore,
\begin{align*}
&\left(\dashint_{B(x_0,R)} f^p \mathrm{d}m\right)^{1/p}\le\left(\dashint_{B(x_0,R)} \lvert f-f_{B(x_0,R)}\rvert^p \mathrm{d}m\right)^{1/p}+f_{B(x_0,R)}\\
&\le C_2 \left(\frac{\Psi(R)}{\delta^\alpha V(x_0,R)}\int_{B(x_0,AR)}\mathrm{d}\Gamma(f)\right)^{1/p}+(1+C_2\delta^\kappa) \dashint_{B(x_0,R)}f \mathrm{d}m.
\end{align*}
By choosing $\delta\in(0,1)$ sufficiently small so that $\varepsilon=C_2\delta^\kappa$ satisfies $\frac{1+\varepsilon}{2^{\frac{p-1}{p}}}<1$, we obtain (\ref{eq_weakPI}).
\end{proof}

In the second step, we prove \ref{eq_FK} using the Wolff potential estimates established in our previous work \cite{Yan25d}. To this end, we first reduce the proof of \ref{eq_FK} further by means of a standard Maz'ya decomposition technique.

\begin{lemma}\label{lem_FK}
Assume that there exist $C>0$, $A>1$ such that for any ball $B(x_0,R)$ and any Borel subset $K\subseteq B(x_0,R)$, we have
\begin{equation}\label{eq_mcap}
m(K)\le C \Psi(R) \mathrm{cap}(K,X\backslash B(x_0,AR)).
\end{equation}
Then \ref{eq_FK} holds.
\end{lemma}

\begin{proof}
For any ball $B(x_0,R)$ and any $u\in \mathcal{F}(B(x_0,R))$, by the $\mathcal{E}_1$-denseness of $\mathcal{F}\cap C_c(B(x_0,R))$, we may assume that $u\in \mathcal{F}\cap C_c(B(x_0,R))$. Then we have
\begin{align*}
\int_{B(x_0,R)}\lvert u\rvert^p \mathrm{d}m=\sum_{n\in \mathbb{Z}}\int_{\left\{2^{n-1}<\lvert u\rvert\le 2^n\right\}}\lvert u\rvert^p \mathrm{d}m\le \sum_{n\in \mathbb{Z}}2^{pn} m \left(\left\{\lvert u\rvert>2^{n-1}\right\}\right).
\end{align*}
For any $n\in \mathbb{Z}$, we have $\{\lvert u\rvert>2^{n-1}\}\subseteq B(x_0,R)$ is open, then by assumption, we have
$$m \left(\left\{\lvert u\rvert>2^{n-1}\right\}\right)\le C\Psi(R)\mathrm{cap}\left(\{\lvert u\rvert>2^{n-1}\},X\backslash B(x_0,AR)\right).$$
Let $v=\left(\frac{\lvert u\rvert-2^{n-2}}{2^{n-2}}\vee0\right)\wedge 1$, then $v\in \mathcal{F}\cap C_c(B(x_0,R))\subseteq \mathcal{F}(B(x_0,AR))$ satisfies $v=1$ on $\{\lvert u\rvert>2^{n-1}\}$, hence by the strongly local property,
$$\mathrm{cap}\left(\{\lvert u\rvert>2^{n-1}\},X\backslash B(x_0,AR)\right)\le \mathcal{E}(v)\le \frac{1}{2^{p(n-2)}}\int_{\left\{2^{n-2}<\lvert u\rvert\le 2^{n-1}\right\}}\mathrm{d}\Gamma(u).$$
Therefore,
\begin{align*}
\int_{B(x_0,R)}\lvert u\rvert^p \mathrm{d}m\le 2^{2p}C\Psi(R)\sum_{n\in \mathbb{Z}}\int_{\left\{2^{n-2}<\lvert u\rvert\le 2^{n-1}\right\}}\mathrm{d}\Gamma(u)=2^{2p}C\Psi(R)\int_{B(x_0,R)}\mathrm{d}\Gamma(u).
\end{align*}
\end{proof}

Now it remains to prove that the following result.

\begin{proposition}\label{prop_mcap}
Assume that all the conditions in Theorem \ref{thm_main} are satisfied. Then (\ref{eq_mcap}) holds.
\end{proposition}

We need the following result for preparation; see \cite[Proposition 3.3]{Yan25d} for the definition of Riesz measures associated with superharmonic functions.

\begin{lemma}\label{lem_lambda}
Let $U$ be a bounded open subset of $X$. Let $u\in\mathcal{F}(U)$ be superharmonic in $U$, and let $\mu$ be the corresponding Riesz measure. Then for any $\lambda>0$, we have
$$\lambda^{p-1}\mathrm{cap}\left(\left\{u>\lambda\right\},X\backslash U\right)\le \mu(U).$$
\end{lemma}

\begin{proof}
We follow the argument used in the proof of \cite[LEMMA 3.9]{KM94Acta}. On the one hand, since $\frac{u}{\lambda}\wedge 1\in \mathcal{F}(U)$ and $\frac{\widetilde{u}}{\lambda}\wedge 1=1$ q.e. on $\{u>\lambda\}$, we have
$$\mathrm{cap}\left(\left\{u>\lambda\right\},X\backslash U\right)\le \mathcal{E}(\frac{u}{\lambda}\wedge 1)=\frac{1}{\lambda^p}\mathcal{E}(u\wedge \lambda).$$
On the other hand, since $u\wedge \lambda\in \mathcal{F}(U)$, by the strongly local property, we have
$$\mathcal{E}(u\wedge \lambda)=\mathcal{E}(u;u\wedge \lambda)=\int_U \left(\widetilde{u}\wedge \lambda\right) \mathrm{d}\mu\le \lambda\mu(U).$$
Therefore,
$$\mathrm{cap}\left(\left\{u>\lambda\right\},X\backslash U\right)\le \frac{1}{\lambda^p}\mathcal{E}(u\wedge \lambda)\le \frac{1}{\lambda^{p-1}}\mu(U).$$
\end{proof}

We recall the following Wolff potential estimates from \cite{Yan25d}. Notice that, in \cite{Yan25d}, assumptions in addition to \ref{eq_VD}, \ref{eq_EHI}, and \hyperlink{eq_cap}{$\text{cap}(\Psi)$} are used to establish these estimates.

\begin{lemma}[{\cite[Theorem 3.12]{Yan25d}}]\label{lem_Wolff}
Assume that all the conditions in Theorem \ref{thm_main} are satisfied. Then there exists $C>0$ such that for any bounded open set $U$, for any $u\in\mathcal{F}$ which is non-negative superharmonic in $U$, for any Lebesgue point $x_0\in U$ of $u$, for any $R>0$ with $B(x_0,4R)\subseteq U$, we have
\begin{equation*}
\frac{1}{C}\mathcal{W}^{\mu[u]}(x_0,R)\le u(x_0)\le C \left(\einf_{B(x_0,R)}u+\mathcal{W}^{\mu[u]}(x_0,2R)\right),
\end{equation*}
where $\mu[u]$ is the Riesz measure associated with $u$ in $U$, and
$$\mathcal{W}^\mu(x,R)=\sum_{n=0}^{+\infty} \left(\frac{\mu(B(x,\frac{1}{2^n}R))}{\mathrm{cap}(B(x,\frac{1}{2^{n+1}}R),X\backslash B(x,\frac{1}{2^n}R))}\right)^{{1}/{(p-1)}},$$
is the Wolff potential of a measure $\mu$.
\end{lemma}

We now prove Proposition \ref{prop_mcap}, which completes the proof of Theorem \ref{thm_main}.

\begin{proof}[Proof of Proposition \ref{prop_mcap}]
Let $B(x_0,R)$ be a ball and let $K\subseteq B(x_0,R)$ be a Borel set. By the regularity of $m$ and $\mathrm{cap}(\cdot,X\backslash B(x_0,16R))$, it suffices to prove that for any compact $K\subseteq B(x_0,R)$, we have
$$m(K)\lesssim \Psi(R)\mathrm{cap}(K,X\backslash B(x_0,16R)).$$

Let $u$ and $\mu$ be the capacitary potential and the capacitary measure, respectively, associated with $\left(K,X\backslash B(x_0,16R)\right)$, as in \cite[Lemma 3.8, Remark 3.9]{Yan25d}, then $u\in \mathcal{F}(B(x_0,16R))$ is non-negative superharmonic in $B(x_0,16R)$, $0\le u\le 1$ in $X$, $\widetilde{u}=1$ q.e. on $K$, $\mathrm{supp}(\mu)\subseteq K$, and
\begin{equation}\label{eq_mcap1}
\mu(B(x_0,16R))=\mu(K)=\mathcal{E}(u)=\mathrm{cap}(K,X\backslash B(x_0,16R)).
\end{equation}
By Lemma \ref{lem_Wolff}, for $m$-a.e. $x\in K$, we have
\begin{equation}\label{eq_mcap2}
1\lesssim \einf_{B(x,2R)}u+\mathcal{W}^\mu(x,4R),
\end{equation}
where $\einf_{B(x,2R)}u\in[0,1]$, and by \ref{eq_VD}, \hyperlink{eq_cap}{$\text{cap}(\Psi)$}, we have
\begin{equation}\label{eq_mcap3}
\mathcal{W}^\mu(x,4R)\asymp\sum_{n=-2}^{+\infty} \left(\frac{\mu(B(x,\frac{1}{2^n}R))}{V(x,\frac{1}{2^n}R)}\Psi(\frac{1}{2^n}R)\right)^{1/(p-1)}.
\end{equation}

We claim that
\begin{equation}\label{eq_mcap4}
1\lesssim \sum_{n=-2}^{+\infty} \left(\frac{\mu(B(x,\frac{1}{2^n}R))}{V(x,\frac{1}{2^n}R)}\Psi(\frac{1}{2^n}R)\right)^{1/(p-1)}.
\end{equation}
Indeed, if $\einf_{B(x,2R)}u=0$, then it follows directly from (\ref{eq_mcap2}) and (\ref{eq_mcap3}). Assume $\lambda=\einf_{B(x,2R)}u\in(0,1]$, then for any $\varepsilon\in(0,\lambda)$, we have
$$K\subseteq B(x_0,R)\subseteq B(x,2R)\subseteq \{u>\lambda-\varepsilon\}\subseteq B(x_0,16R).$$
By Lemma \ref{lem_lambda} and (\ref{eq_mcap1}), we have
\begin{align*}
&\mathrm{cap}\left(B(x_0,R),X\backslash B(x_0,16R)\right)\le\mathrm{cap}\left(\{u>\lambda-\varepsilon\},X\backslash B(x_0,16R)\right)\\
&\le \frac{1}{(\lambda-\varepsilon)^{p-1}}\mu(B(x_0,16R))=\frac{1}{(\lambda-\varepsilon)^{p-1}}\mu(B(x,2R)).
\end{align*}
Letting $\varepsilon\downarrow0$, by \ref{eq_VD}, \hyperlink{eq_cap}{$\text{cap}(\Psi)$}, we have
\begin{align*}
&\einf_{B(x,2R)}u=\lambda\\
&\le \left(\frac{\mu(B(x,2R))}{\mathrm{cap}\left(B(x_0,R),X\backslash B(x_0,16R)\right)}\right)^{1/(p-1)}\lesssim \left(\frac{\mu(B(x,2R))}{V(x,2R)}\Psi(2R)\right)^{1/(p-1)}.
\end{align*}
Combining this with (\ref{eq_mcap2}) and (\ref{eq_mcap3}), we also have (\ref{eq_mcap4}).

By (\ref{eq_mcap4}), there exist $C_1>0$ and $K_0\subseteq K$ with $m(K_0)=m(K)$ such that, for any $x\in K_0$, we have
$$\sum_{n=-2}^{+\infty} \left(\frac{\mu(B(x,\frac{1}{2^n}R))}{V(x,\frac{1}{2^n}R)}\Psi(\frac{1}{2^n}R)\right)^{1/(p-1)}\ge \frac{1}{C_1}.$$
Let $C_2=\left(\frac{2^{\frac{3p}{p-1}}C_1}{2^{\frac{p}{p-1}}-1}\right)^{p-1}$, then there exists $n_x\ge-2$ such that
$$\frac{\mu(B(x,\frac{1}{2^{n_x}}R))}{V(x,\frac{1}{2^{n_x}}R)}\Psi(\frac{1}{2^{n_x}}R)\ge \frac{1}{C_22^{pn_x}}.$$
Otherwise, for any $n\ge-2$, we have
$$\frac{\mu(B(x,\frac{1}{2^n}R))}{V(x,\frac{1}{2^n}R)}\Psi(\frac{1}{2^n}R)<\frac{1}{C_22^{pn}},$$
which gives
$$\sum_{n=-2}^{+\infty} \left(\frac{\mu(B(x,\frac{1}{2^n}R))}{V(x,\frac{1}{2^n}R)}\Psi(\frac{1}{2^n}R)\right)^{1/(p-1)}<\sum_{n=-2}^{+\infty}\frac{1}{C_2^{\frac{1}{p-1}}2^{\frac{p}{p-1}n}}=\frac{1}{C_2^{\frac{1}{p-1}}}\frac{2^{\frac{3p}{p-1}}}{2^{\frac{p}{p-1}}-1}=\frac{1}{C_1},$$
contradiction. Hence
$$V(x,\frac{1}{2^{n_x}}R) \le C_22^{pn_x}\Psi(\frac{1}{2^{n_x}}R)\mu(B(x,\frac{1}{2^{n_x}}R)).$$

By \ref{eq_VD}, \hyperlink{eq_cap}{$\text{cap}(\Psi)$}, applying \cite[Proposition 2.1, Remark 2.2]{Yan25a}, there exists $C_3>0$ such that
$$\frac{\Psi(R)}{\Psi(\frac{1}{2^{n_x}}R)}\ge \frac{1}{C_3}\left(\frac{R}{\frac{1}{2^{n_x}}R}\right)^p=\frac{2^{pn_x}}{C_3},$$
hence
$$V(x,\frac{1}{2^{n_x}}R)\le C_2C_3\Psi(R)\mu(B(x,\frac{1}{2^{n_x}}R)).$$
By the $5B$-covering lemma (see \cite[Theorem 1.2]{Hei01}), there exists a countable family of disjoint balls $\{B(x_k,\frac{1}{2^{n_k}}R)\}_k$ such that
$$K_0\subseteq \bigcup_{k}B(x_k,5\cdot\frac{1}{2^{n_k}}R).$$
Therefore, by \ref{eq_VD},
\begin{align*}
&m(K)=m(K_0)\le \sum_k V(x_k,5\cdot\frac{1}{2^{n_k}}R)\le C_{VD}^3\sum_k V(x_k,\frac{1}{2^{n_k}}R)\\
&\le C_{VD}^3C_2C_3\Psi(R)\sum_{k}\mu(B(x_k,\frac{1}{2^{n_k}}R))=C\Psi(R)\mu \left(\bigcup_kB(x_k,\frac{1}{2^{n_k}}R)\right)\\
&\le C\Psi(R)\mu(B(x_0,16R))=C\Psi(R)\mathrm{cap}(K,X\backslash B(x_0,16R)),
\end{align*}
where $C=C_{VD}^3C_2C_3$, which gives (\ref{eq_mcap}).
\end{proof}

\bibliographystyle{plain}
%\bibliography{/Users/meng/Dropbox/myref}

\end{document}